\documentclass[10pt]{amsart}

\usepackage{amssymb,bbm}

\usepackage[breaklinks=true]{hyperref}

\newtheorem{lem}{Lemma}[section]
\newtheorem{thrm}[lem]{Theorem}
\newtheorem{prop}[lem]{Proposition}

\theoremstyle{definition}
\newtheorem{defn}[lem]{Definition}

\theoremstyle{remark}

\newcommand{\R}{\mathbb{R}}
\newcommand{\C}{\mathbb{C}}

\newcommand{\N}{\mathbb{N}}

\newcommand{\Schwartz}{\mathcal{S}}

\renewcommand{\Re}{\operatorname{Re}}
\renewcommand{\Im}{\operatorname{Im}}
\newcommand{\eps}{\varepsilon}
\newcommand{\mc}{\mathcal}

\newcommand{\mr}{\mathrm}

\DeclareMathOperator{\supp}{supp}
\DeclareMathOperator{\sign}{sign}

\newcommand{\<}{\langle}
\renewcommand{\>}{\rangle}
\newcommand{\p}{\partial}

\newcommand{\bbo}{\mathbbm 1}

\newcommand{\qtq}[1]{\quad\text{#1}\quad}

\newcommand{\qt}[1]{\quad\text{#1}}

\newcommand{\cM}{\mc M}
\newcommand{\loc}{\mr{loc}}
\newcommand{\Tmax}{T_\textit{max}}
\newcommand{\I}{I}  
\newcommand{\II}{I\!I}
\newcommand{\III}{I\!I\!I}
\newcommand{\Err}{\textsl{Error}}

\newcounter{smalllist}
\newenvironment{SL}{\begin{list}{\hss\upshape(\roman{smalllist})\hss}{%
\setlength{\topsep}{0mm}\setlength{\parsep}{0mm}\setlength{\itemsep}{0mm}%
\setlength{\labelwidth}{1.75em}\setlength{\labelsep}{\the\fontdimen2\font}\setlength{\leftmargin}{\the\labelwidth}\addtolength{\leftmargin}{\labelsep}%
\setlength{\itemindent}{0em}\usecounter{smalllist}%
}}{\end{list}}

\title[Scattering for concentrated nonlinearities]{Scattering for the nonlinear Schr\"odinger equation with concentrated nonlinearity}
\author[B.~Harrop-Griffiths]{Benjamin Harrop-Griffiths}
\address{Department of Mathematics \& Statistics, Georgetown University\\Washington, DC 20057, USA}
\email{benjamin.harropgriffiths@georgetown.edu}
\thanks{}

\author[R.~Killip]{Rowan Killip}
\address{Department of Mathematics, University of California, Los Angeles\\ CA 90095, USA}
\email{killip@math.ucla.edu}
\thanks{}

\author[M.~Vi\c san]{Monica Vi\c san}
\address{Department of Mathematics, University of California, Los Angeles\\CA 90095, USA}
\email{visan@math.ucla.edu}
\thanks{}

\numberwithin{equation}{section}
\allowdisplaybreaks

\begin{document}

\begin{abstract}
We consider the cubic defocusing nonlinear Schr\"odinger equation in one dimension with the nonlinearity concentrated at a single point.  We prove global well-posedness in the scaling-critical space $L^2(\R)$ and scattering for all such solutions.  Moreover, we demonstrate that the same phenomenology holds whenever nonlinear effects are sufficiently concentrated in space.
\end{abstract}

\maketitle

\section{Introduction}

We consider the nonlinear Schr\"odinger equation
\begin{equation}\label{g NLS}\tag{NLS$_g$}
i\psi_t = - \psi_{xx} + g(x)|\psi|^2\psi
\end{equation}
in one space dimension.  This equation describes the time evolution of a `wave function' $\psi(t)\colon\R\to\C$.  Notice that we allow the coupling $g(x)$ of the nonlinearity to vary in space; indeed, we wish to discuss the case where $g(x)$ is strongly concentrated around a single point.  To this end, we fix $g\in L^1(\R)$ satisfying $\int_\R g(x) \,dx =1$ and then consider the rescaled coupling constants
\begin{equation}\label{g choice}
 g_\eps(x) := \tfrac1\eps g\bigl(\tfrac x\eps\bigr)
\end{equation}
with $\eps>0$ small.

Evidently, $\int_\R g_\eps =1$ for all $\eps>0$.  The fact that this integral is positive is crucial for our analysis; it gives the model a defocusing flavor.  Notice however that we do \emph{not} demand that $g\geq0$ pointwise.  We have chosen that $\int_\R g =1$ for simplicity; other values can be accommodated by a simple rescaling.

Sending $\eps\downarrow 0$, leads to the point-concentrated nonlinear Schr\"odinger equation,
\begin{equation}\label{point NLS}\tag{NLS$_{\delta}$}
i\psi_t = - \psi_{xx} + \delta(x)|\psi|^2\psi,
\end{equation}
in which the coupling coefficient has become the Dirac delta function.   A considerable theory has been developed around this equation in recent years; see the review article 
\cite{MR4653819}.  In particular, the seemingly naive $\eps\downarrow 0$ limit just described has been shown to hold rigorously on compact time intervals for sufficiently regular initial data \cite{MR3275343,MR2318828}; specifically, these papers demand initial data $\psi_0\in H^1(\R)$.

The problem we analyze in this paper combines three challenges: we wish to describe the \emph{long-time} behavior for \emph{large} initial data at \emph{critical} regularity.

Any intrinsic notion of size must naturally respect the scaling symmetry of \eqref{point NLS}, namely,
\begin{equation}\label{mass crit}
\psi(t,x)\mapsto \sqrt\lambda \psi(\lambda^2 t,\lambda x)\qt{for \(\lambda>0\)}.
\end{equation}
This scaling also leaves the mass functional
\begin{equation}\label{mass}
M[\psi] = \int_\R|\psi|^2\,dx
\end{equation}
invariant, thereby identifying this model as \emph{mass-critical}.  The mass is also conserved under the \eqref{point NLS} dynamics (we will verify this later).  In this way, we are directed toward the most natural (and most ambitious) class of initial data: $L^2(\R)$.

Local well-posedness of \eqref{point NLS} in $L^2(\R)$ was shown previously in \cite{MR4188177}; moreover, global well-posedness and scattering have been shown under suitable smallness assumptions in \cite{MR4188177,MR4548487}.  Our main result addresses arbitrarily large initial data:

\begin{thrm}\label{t:point GWP}
The equation \eqref{point NLS} is globally well-posed in $L^2(\R)$.  Moreover,
\begin{equation}\label{STB:1}
\int_\R |\psi(t,0)|^4\,dt \lesssim \| \psi(0) \|_{L^2}^4
\end{equation}
and every solution scatters, that is, there are asymptotic states \(\psi_\pm\in L^2(\R)\) so that
\begin{equation}\label{scat defn}
\| \psi(t) - e^{it\partial_x^2} \psi_\pm \|_{L^2(\R)} \to 0 \qtq{as} t \to \pm\infty.
\end{equation}
\end{thrm}

This result will be proven in Section~\ref{S:3} building on the local theory developed in Section~\ref{s:local}.  The crucial ingredient for treating large data globally in time is the new spacetime estimate provided by Proposition~\ref{p:identity}.  This estimate expresses the defocusing nature of our equation in a manner that is reminiscent of Morawetz-type identities.  However, the estimate differs from these well-known monotonicity formulas by expressing a direct inequality between the linear and nonlinear flows.  Moreover, the proof is also very different; it is not centered on a quantity that is monotone in time.

Earlier we motivated the consideration of the point-concentrated nonlinearity as the limit of \eqref{g NLS} with nonlinear coupling  $g_\eps(x)$ defined in \eqref{g choice}.  In Section~\ref{S:4} we return to this model and show that solutions of this model closely follow those of \eqref{point NLS} provided $\eps$ is sufficiently small, depending on the initial data.  (A simple scaling argument shows that we cannot choose $\eps$ independently of the initial data.)  The precise formulation of our result is as follows:

\begin{thrm}\label{t:perturbations}
Fix $g\in L^1(\R)$ with $\int_\R g =1$. Let $\psi^\eps(t)$ denote the solutions to
\begin{equation}\label{g epsilon NLS}\tag{NLS$_{g_\epsilon}$}
i\psi^\eps_t = - \psi^\eps_{xx} + g_\eps(x)|\psi^\eps|^2\psi^\eps,
\end{equation}
with initial data \(\psi^\eps(0)\to \psi(0)\) in \(L^2(\R)\).  There exists \(\eps_0>0\) so that for all \(0<\eps\leq \eps_0\), the equation \eqref{g epsilon NLS} is globally well-posed and scatters in \(L^2(\R)\). Moreover,
\begin{equation}\label{g converg}
\lim_{\eps \to 0}\ \sup_{t\in \R}\ \|\psi^\eps(t) - \psi(t)\|_{L^2} = 0,
\end{equation}
where $\psi$ denotes the solution to \eqref{point NLS} with initial data $\psi(0)$.
\end{thrm}

\subsection*{Acknowledgements}B.~H.-G. was supported by NSF grant DMS-2406816,  R.~K. was supported by NSF grant DMS-2154022, and M.~V. was supported by NSF grants DMS-2054194 and DMS-2348018.

\section{Preliminaries and local theory}\label{s:local}

We start with notation.  Our convention for the inner product on $L^2(\R)$ is
\[
\langle g, f\rangle = \int_\R \overline{g(x)} f(x)\,dx.
\]

Given a Banach space \(X\) and an interval \(I\), we write \(C(I;X)\) for the space of bounded continuous functions $\psi:I\to X$ equipped with the supremum norm. The omission of $X$ indicates \(X = \C\); similarly, the omission of $I$ indicates $I=\R$.

We write \(\cM\) to denote the space of finite signed Borel measures on \(\R\), equipped with the total variation norm. We recall that \(\cM\) is dual to the closed subspace \(C^0(\R)\subseteq C(\R)\) of continuous functions vanishing at infinity.

If \(\mu\in \cM\) and \(1\leq p<\infty\), we take \(L_\mu^p\) to be the space of (equivalence classes of) Borel-measurable \(\psi\colon \R\to \C\) with finite norm
\[
\|\psi\|_{L_\mu^p}^p = \int_\R |\psi(x)|^p\,d|\mu|(x).
\]
For any \(1\leq p\leq \infty\), we have 
\begin{equation}\label{trivial/useful}
\|\psi\|_{L_\mu^p}\leq \|\mu\|_{\cM}^{\frac1p}\|\psi\|_C^{ }
\end{equation}
and so by duality,
\begin{equation}\label{trivial/useful dual}
\|h \mu\|_{\cM}\leq \|\mu\|_{\cM}^{\frac1p}\|h\|_{L_\mu^{p'}},
\end{equation}
where \(\frac1p + \frac1{p'} = 1\).

When using mixed spacetime norms, we will use subscripts to indicate the corresponding variable, e.g., \(C^{}_tL_x^2 = C(\R;L^2(\R))\). If \(I\subseteq \R\) is an interval, we use the subscript \(I\) to indicate that the time domain is \(I\) instead of \(\R\), e.g., \(C_IL_x^2 = C(I;L^2(\R))\).

Our first result is a small refinement of the endpoint Strichartz estimate, drawing attention to continuity and decay at infinity.  

\begin{lem}\label{l:linear}
For any $\psi_0\in L^2(\R)$ and any forcing term $f\in L_t^{\frac43}\cM$, 
\begin{equation}\label{linear Duhamel}
\psi(t) := e^{it\p_x^2}\psi_0 - i \int_0^t e^{i(t-s)\p_x^2}f(s)\,ds
\end{equation}
satisfies the  Strichartz bounds
\begin{equation}\label{linear Strichartz}
\|\psi\|_{C^{}_tL_x^2 \cap L_t^4C^0_x} \lesssim \|\psi_0\|_{L^2} + \|f\|_{L_t^{4/3}\cM} .
\end{equation}
\end{lem}
\begin{proof}
The classical Strichartz estimate from \cite{MR741079} shows that
\begin{equation}\label{classical Strichartz}
\|e^{it\p_x^2}\psi_0\|_{C^{}_tL_x^2\cap L_t^4L_x^\infty}\lesssim  \|\psi_0\|_{L^2}.
\end{equation}

By virtue of the Sobolev embedding \(H^1(\R)\hookrightarrow C^0(\R)$, we have
\[
\|e^{it\p_x^2}\psi_0\|_{C^0_x}\lesssim \|e^{it\p_x^2}\psi_0\|_{H^1}\lesssim \|\psi_0\|_{H^1} \quad\text{for all $t\in \R$}.
\]
Thus,  for any \(\psi_0\in H^1\), we may upgrade \eqref{classical Strichartz} to
\begin{equation}\label{slightly better Strichartz}
\|e^{it\p_x^2}\psi_0\|_{C^{}_tL_x^2\cap L_t^4C^0_x}\lesssim \|\psi_0\|_{L^2}.
\end{equation}
The restriction that $\psi_0\in H^1$ is then easily removed by exploiting the density of \(H^1(\R)\) in \(L^2(\R)\).

Recalling that \(\cM\) is dual to $C^0$, we deduce the dual estimate
\begin{equation}\label{dual Strichartz}
\biggl\|\int_\R e^{-it\p_x^2}f(t)\,dt\biggr\|_{L^2} \lesssim  \|f\|_{L_t^{4/3}\cM_x}.
\end{equation}
The proof of \eqref{linear Strichartz} is now completed by combining \eqref{slightly better Strichartz} and \eqref{dual Strichartz} with the Christ--Kiselev lemma \cite{MR1809116}.
\end{proof}

We now turn to the problem of well-posedness. Here, we will treat the rather more general model
\begin{equation}\label{mu NLS}\tag{NLS$_\mu$}
i\psi_t = - \psi_{xx} + |\psi|^2\psi\,d\mu,
\end{equation}
where \(\mu\in \cM\).  We will only consider strong solutions to this equation:
 
\begin{defn}
Let \(I\subseteq \R\) be an interval. We say \(\psi\in C_IL_x^2\) is a \emph{solution} to \eqref{mu NLS} if \(\psi\in L_{I,\loc}^4C_x\) and for some \(t_0\in I\) and all \(t\in I\) we have
\begin{equation}\label{Duhamel}
\psi(t) = e^{i(t-t_0)\p_x^2}\psi(t_0) - i\int_{t_0}^t e^{i(t-s)\p_x^2}\bigl[|\psi(s)|^2\psi(s)\,d\mu\bigr]\,ds.
\end{equation}
\end{defn}

Notice that the estimate \eqref{linear Strichartz} ensures that \eqref{Duhamel} makes sense for \(\psi\in C_IL_x^2\cap L_{I,\loc}^4C_x\). Moreover, a straightforward computation shows that solutions of \eqref{mu NLS} satisfy \eqref{Duhamel} for \emph{all} choices of \(t_0\in I\).

One may also easily verify the following time translation and reversal symmetries of \eqref{mu NLS}: If $\psi(t)$ is a solution to \eqref{mu NLS}, then so too are \(t\mapsto\psi(t+t_0)\) and \(t \mapsto \overline{\psi(-t)}\).  These symmetries allow us to focus our development of the local well-posedness theory to time intervals of the form $[0,T)$.

\begin{thrm}[Local theory for \eqref{mu NLS}]\label{t:LWP}\leavevmode
\begin{SL}
\item\label{LWP:E} For any $\psi_0\in L^2(\R)$, there is a $T>0$ and a solution $\psi:[0,T)\to L^2(\R)$ to \eqref{mu NLS} with $\psi(0)=\psi_0$.
\item\label{LWP:U} If $\psi,\phi$ are solutions to \eqref{mu NLS} that agree at one time, then they agree throughout their common domain.  
\item\label{LWP:SC} The maximal existence interval (forward in time) has the form $[0,\Tmax(\psi_0))$.  If the maximal solution satisfies
\begin{equation}\label{scattering criterion}
\int_0^{\Tmax} \int_\R |\psi(t,x)|^4\,d|\mu|(x)\,dt < \infty,
\end{equation}
then $\Tmax=\infty$ and the solution scatters forward in time {\upshape(}cf. \eqref{scat defn}{\upshape)}.
\item\label{LWP:SD} There exists $\eta_0>0$, depending only on $\mu$, so that \(\|\psi_0\|_{L^2}\leq \eta_0\) implies \eqref{scattering criterion}.
\item\label{LWP:CD}  If $\psi_{n,0}\to\psi_0$ in $L^2(\R)$, then the corresponding solutions $\psi_n(t)$ converge to $\psi(t)$ in $C^{}_tL^2\cap L^4_t C^{0}_x$ on any time interval $[0,T]$ with $T<\Tmax(\psi_0)$.  In particular, the mapping $\psi_0\to \Tmax(\psi_0)$ is lower semicontinuous. 
\item\label{LWP:MC} Mass is conserved under the flow: $M[\psi(t)]=M[\psi_0];$ see \eqref{mass}.
\end{SL}
\end{thrm}

\begin{proof}[Proof of {\upshape\ref{LWP:E}: Existence.}]
Given \(\eta>0\) and \(\psi_0\in L^2(\R)\), we may combine \eqref{linear Strichartz} with monotone convergence to find \(T = T(\eta,\psi_0)>0\) so that
\begin{equation}\label{small init condition}
\|e^{it\p_x^2}\psi_0\|_{L_{[0,T)}^4C_x}\leq \eta.
\end{equation}
We then define the closed ball
\[
B = \bigl\{\psi\in L_{[0,T)}^4L_\mu^4:\|\psi\|_{L_{[0,T)}^4L_\mu^4}\leq 2\eta\|\mu\|_{\cM}^{\frac14}\bigr\},
\]
and consider the mapping
\[
\Phi[\psi](t) = e^{it\p_x^2}\psi_0 - i\int_0^t e^{i(t-s)\p_x^2}\bigl[|\psi(s)|^2\psi(s)\,d\mu\bigr]\,ds.
\]

Applying \eqref{trivial/useful}, \eqref{trivial/useful dual}, and \eqref{linear Strichartz}, for \(\psi\in B\) we may estimate
\begin{align*}
\|\Phi[\psi]\|_{L_{[0,T)}^4L_\mu^4} &\leq \|\mu\|_{\cM}^{\frac14}\|\Phi[\psi]\|_{L_{[0,T)}^4C_x}\\
&\leq \|\mu\|_{\cM}^{\frac14}\Bigl[\|e^{it\p_x^2}\psi_0\|_{L_{[0,T)}^4C_x} + A\||\psi|^2\psi\mu\|_{L_{[0,T)}^{4/3}\cM}\Bigr]\\
&\leq \|\mu\|_{\cM}^{\frac14}\Bigl[\eta + A\|\mu\|_{\cM}^{\frac14}\|\psi\|_{L_{[0,T)}^4L_\mu^4}^3\Bigr]\\
&\leq \|\mu\|_{\cM}^{\frac14}\Bigl[\eta + 8A\eta^3\|\mu\|_{\cM}\Bigr],
\end{align*}
for some absolute constant \(A>0\) dictated by the Strichartz inequality \eqref{linear Strichartz}. Similarly, for \(\psi,\phi\in B\) we may bound the difference
\[
\|\Phi[\psi] - \Phi[\phi]\|_{L_{[0,T)}^4L_\mu^4}\leq 24A\eta^2\|\mu\|_{\cM}\|\psi - \phi\|_{L_{[0,T)}^4L_\mu^4}.
\]

For \(\eta>0\) sufficiently small, depending only on \(\|\mu\|_{\cM}\), these estimates show that \(\Psi\colon B\to B\) and that this mapping is a contraction.  Thus, there is a unique \(\psi\in B\) satisfying \(\psi = \Phi[\psi]\).  Applying \eqref{trivial/useful dual} and \eqref{linear Strichartz} once again, we see that
\begin{align}\label{back out bound}
\|\psi\|_{C_{[0,T)}L_x^2\cap L_{[0,T)}^4C_x}&\lesssim \|\psi_0\|_{L^2} + \||\psi|^2\psi\mu\|_{L_{[0,T)}^{4/3}\cM}\\&\lesssim \|\psi_0\|_{L^2} + \|\mu\|_{\cM}^{\frac14}\|\psi\|_{L_{[0,T)}^4L_\mu^4}^3 <\infty,\notag
\end{align}
and so \(\psi\in C_{[0,T)}L_x^2\cap L_{[0,T)}^4C_x\) is a solution to \eqref{mu NLS} satisfying \(\psi(0) = \psi_0\).
\end{proof}

\begin{proof}[Proof of {\upshape\ref{LWP:U}: Uniqueness.}]
Suppose for a contradiction that two solutions $\psi$ and $\phi$ agree at some time $t_0$, but disagree at another time $t_1>t_0$.  As both solutions are $L^2_x$-continuous and belong to \(L_{\loc}^4C_x\), we may choose such a pair of times so close together that for any prescribed $\eta>0$,
\[
\|\psi\|_{L_I^4L_\mu^4} + \|\phi\|_{L_I^4L_\mu^4}\leq \eta
\]
on the interval \(I = [t_0,t_1]\).  Arguing as in our proof of existence, we may then apply the estimates \eqref{trivial/useful}, \eqref{trivial/useful dual}, and \eqref{linear Strichartz} to the Duhamel formula \eqref{Duhamel} to bound
\[
\|\psi - \phi\|_{C_IL^2_x\cap L_I^4L_\mu^4}\leq 3 A \eta^2\|\mu\|_{\cM}\|\psi - \phi\|_{L_I^4L_\mu^4},
\]
where the constant \(A\) is chosen as before. When \(\eta\) is taken to be sufficiently small, depending only on \(\|\mu\|_{\cM}\), we deduce the desired contradiction.
\end{proof}

\begin{proof}[Proof of {\upshape\ref{LWP:SC}: Blowup/scattering criterion}]
If $\Tmax$ were contained in the interval of existence, then we could apply our existence theorem to extend the solution beyond this time; this would contradict the maximality of the solution.

Our next goal is to show that if \eqref{scattering criterion} holds, then   $e^{-it\p_x^2}\psi(t)$ is $L^2$-Cauchy as \(t\uparrow \Tmax\).  When \(\Tmax=\infty\), this immediately yields scattering.  If instead \(\Tmax<\infty\), then we could extend $\psi(t)$ continuously to $t=\Tmax$, contradicting the first paragraph of this proof.

For any pair of times $t_0,t_1\in[T,\Tmax)$, we may apply \eqref{linear Strichartz} to the identity \eqref{Duhamel} to obtain
\begin{align*}
\|e^{-it_1\p_x^2}\psi(t_1) - e^{-it_0\p_x^2}\psi(t_0)\|_{L_x^2} 
&\lesssim \bigl\||\psi|^2\psi\,d\mu\bigr\|_{L_{[T,\Tmax)}^{4/3}\cM}\lesssim \|\mu\|_{\cM}^{\frac14}\|\psi\|_{L_{[T,\Tmax)}^4L^4_\mu}^3.
\end{align*}
From \eqref{scattering criterion} and dominated convergence, we deduce that $e^{-it\p_x^2}\psi(t)$ is indeed $L^2$-Cauchy as \(t\uparrow \Tmax\).
\end{proof}

\begin{proof}[Proof of {\upshape\ref{LWP:SD}: Small-data global bounds.}]
By applying \eqref{linear Strichartz}, we see that if \(\|\psi_0\|_{L^2}\) is sufficiently small then \eqref{small init condition} is satisfied with \(T = \infty\).  Thus, the proof of part (i) gives global existence ($\Tmax=\infty$) and \eqref{scattering criterion} for all sufficiently small initial data.
\end{proof}

\begin{proof}[Proof of {\upshape\ref{LWP:CD}: Continuous dependence.}]
Let $\psi:[0,\Tmax)\to L^2(\R)$ be the maximal forward solution with initial data $\psi_0$.  Our first goal is the following: given $T< \Tmax$, we wish to exhibit $\delta>0$ so that any initial data with
\begin{equation}\label{cont init}
\|\varphi_0 - \psi_0\|_{L^2}\leq \delta
\end{equation}
leads to a solution $\varphi(t)$ throughout the interval $[0,T]$.  In the process, we will obtain a quantitative bound on
\begin{equation}\label{err}
\Err(t):= \|\psi - \varphi\|_{C_{[0,t]}L_x^2\cap L_{[0,t]}^4C_x}
\end{equation}
that allows us to conclude that $\Err(T)\to 0$ as $\varphi_0 \to \psi_0$.

The choice of $T$ ensures that $\psi$ has finite $L^4_tC_x^{}$ norm on $[0,T]$.  Thus, given any $\eta>0$, we may partition this interval into 
\[
J_m := [t_{m-1},t_m] \qtq{with} 0 = t_0 < t_1 < \dots < t_M = T,
\]
in such a way that 
\begin{equation}\label{wee J}
\|\psi\|_{L_{J_m}^4C_x}\leq \eta \qt{for all $m=1,\ldots,M$}.
\end{equation}
Working efficiently, the number of intervals $M=M(\eta)$ can be ensured to satisfy
\begin{equation}\label{N size}
M\leq \tfrac1{\eta^4}\|\psi\|_{L_{[0,T]}^4C_x}^4 + 1.
\end{equation}

We also define a `stopping time' $\tau\in[0,T]$ as the maximal such $\tau$ so that
\begin{equation}\label{boot}
\|\psi - \varphi\|_{L_{[0,\tau]}^4C_x}\leq 2\eta.
\end{equation}
Evidently, $\tau$ depends on $\varphi_0$ and $\eta$.  Part\ref{LWP:SC} shows that the solution $\varphi$ cannot blow up before time $\tau$.  If $\tau\geq t_{m-1}$, then we write $J_m^\tau=[t_{m-1},t_m\wedge\tau]$.

Applying \eqref{linear Strichartz} in \eqref{Duhamel} and using \eqref{wee J} and \eqref{boot}, we conclude that
\[
\|\psi - \varphi\|_{C_{J_m^\tau}L_x^2\cap L_{J_m^\tau}^4C_x}\lesssim \|\psi(t_{m-1}) - \varphi(t_{m-1})\|_{L^2} + \eta^2\|\mu\|_{\cM}\|\psi - \varphi\|_{L_{J_m^\tau}^4C_x},
\]
for each \(m=1,\dots, M\) with $\tau\geq t_{m-1}$.   Thus if \(\eta>0\) is sufficiently small, depending only on \(\|\mu\|_{\cM}\), we have
\[
\Err(t_m\wedge\tau) \leq e^A \Err(t_{m-1}),
\]
for some constant \(A>0\) (that may be different from before). Applying this estimate iteratively for \(m=1,\dots,M\), we may use \eqref{N size} to conclude that
\begin{equation}\label{strap}
\Err(T\wedge\tau) \leq \exp\Bigl(A\Bigl[\tfrac1{\eta^4}\|\psi\|_{L_{[0,T]}^4C_x}^4+1\Bigr]\Bigr)\|\psi_0 - \varphi_0\|_{L^2}.
\end{equation}

Now if \(0<\delta\ll1\) is chosen so small that RHS\eqref{strap} is less than $2\eta$, then a continuity argument yields that $\tau=T$ and \eqref{boot} and holds throughout $[0,T]$.  In this way, we deduce both the existence of the solution $\varphi$ and the Lipschitz bound
\begin{equation}\label{strap cons}
\|\psi - \varphi\|_{C_{[0,T]}L_x^2\cap L_{[0,T]}^4C_x}\lesssim \exp\Bigl(A\Bigl[\tfrac1{\eta^4}\|\psi\|_{L_{[0,T]}^4C_x}^4+1\Bigr]\Bigr)\|\psi_0 - \varphi_0\|_{L^2}.
\end{equation}
valid for any initial data $\varphi_0$ satisfying \eqref{cont init}.
\end{proof}

\begin{proof}[Proof of {\upshape\ref{LWP:MC}: Mass conservation}]
As \(e^{it\p_x^2}\) is unitary, we infer from \eqref{Duhamel} that
\begin{align*}
\|\psi(t)\|_{L^2}^2 &= \biggl\|\psi_0 - i\int_0^t e^{-is\p_x^2}\bigl[\psi(s)|^2\psi(s)\,d\mu\bigr]\,ds\biggr\|_{L^2}^2\\
&= \|\psi_0\|_{L^2}^2 + 2\Im \biggl\<\psi_0,\int_0^te^{-is\p_x^2}\bigl[|\psi(s)|^2\psi(s)\,d\mu\bigr]\,ds\biggr\>\\
&\quad + \biggl\|\int_0^te^{-is\p_x^2}\bigl[|\psi(s)|^2\psi(s)\,d\mu\bigr]\,ds\biggr\|_{L^2}^2.
\end{align*}
Employing the dual Strichartz estimate \eqref{dual Strichartz} to justify the inner products, we may write the second and third terms as follows:
\begin{align*}
&2\Im \biggl\<\psi_0,\int_0^te^{-is\p_x^2}\bigl[|\psi(s)|^2\psi(s)\,d\mu\bigr]\,ds\biggr\> = 2\Im \int_0^t \bigl\<e^{is\p_x^2}\psi_0,|\psi(s)|^2\psi(s)\,d\mu\bigr\>\,ds,\\
&\biggl\|\int_0^te^{-is\p_x^2}\bigl[|\psi(s)|^2\psi(s)\,d\mu\bigr]\,ds\biggr\|_{L^2}^2 \\
&\qquad= 2\Re\int_0^t\biggl\<  \int_0^s e^{i(s-\sigma)\p_x^2}\bigl[|\psi(\sigma)|^2\psi(\sigma)\,d\mu\bigr] \,d\sigma,|\psi(s)|^2\psi(s)\,d\mu\biggr\>\,ds.
\end{align*}
Combining these and recognizing \eqref{Duhamel}, we get
\[
\|\psi(t)\|_{L^2}^2 = \|\psi_0\|_{L^2}^2 + 2\Im \int_0^t \bigl\<\psi(s),|\psi(s)|^2\psi(s)\,d\mu\bigr\>\,ds = \|\psi_0\|_{L^2}^2,
\]
which is precisely the conservation of mass.
\end{proof}

\section{Global solutions and scattering for a point nonlinearity}\label{S:3}

For $t\neq 0$, the fundamental solution of the linear Schr\"odinger equation is
\begin{equation}\label{fundamental solution}
k(t,x) := (e^{it\p_x^2}\delta)(x) = \tfrac1{\sqrt{4\pi |t|}}e^{-\frac{|x|^2}{4it} - \frac{i\pi}4\sign(t)} .
\end{equation}
Two properties of this kernel that play an essential role in our analysis are
\begin{equation}\label{key}
\overline{ k(t,0)} = k(-t,0) =  i\sign(t) k(t,0)  \qt{for all \(t\neq 0\).}
\end{equation}

\begin{prop}\label{p:identity}
If \(I\subseteq \R\) is an open interval containing zero and \(\psi\in C_IL_x^2\cap L_{I,\loc}^4C_x\) is the strong solution of \eqref{point NLS} satisfying \(\psi(0) = \psi_0\), then
\begin{equation}\label{estimate}
\int_I |\psi(t,0)|^4\,dt \leq \|e^{it\p_x^2}\psi_0\|_{L_I^4C_x}^4.
\end{equation}
\end{prop}

\begin{proof}
We abbreviate the nonlinearity as
\begin{equation}\label{F(t,x)}
F(t,x) = |\psi(t,x)|^2\psi(t,x).
\end{equation}
For any \(0<T\in I\), we employ \eqref{key} to obtain
\begin{align*}
&\int_0^T|\psi(t,0)|^4\,dt\\
&\qquad = \Re\int_0^T\overline{F(t,0)} \psi(t,0)\,dt\\
&\qquad = \Re\int_0^T \overline{F(t,0)}\bigl(e^{it\p_x^2}\psi_0\bigr)(0)\,dt  - \Re i\int_0^T\int_0^t \overline{F(t,0)}k(t-s,0)F(s,0)\,ds\,dt\\
&\qquad = \Re\int_0^T \overline{F(t,0)}\bigl(e^{it\p_x^2}\psi_0\bigr)(0)\,dt\\
&\qquad\quad - \tfrac12\Re \int_0^T\int_0^T \overline{F(t,0)} \, i\sign(t-s)k(t-s,0)F(s,0)\,ds\,dt
\end{align*}
The last step here is symmetrization of the double integral under $s\leftrightarrow t$ employing the first relation in \eqref{key}.  Using the second relation, we see that the resulting double integral is actually sign-definite:
\[
\int_0^T\int_0^T \overline{F(t,0)}k(s-t,0)F(s,0)\,ds\,dt = \left\|\int_0^T e^{-it\p_x^2}\bigl[F(t)\,\delta\bigr]\,dt\right\|_{L^2}^2\geq 0.
\]
Incorporating this information into our earlier computation and using H\"older's inequality, we deduce that
\begin{align*}
\int_0^T|\psi(t,0)|^4\,dt &\leq \Re\int_0^T \overline{F(t,0)}\bigl(e^{it\p_x^2}\psi_0\bigr)(0)\,dt \\
	&\leq \left[\int_0^T |\psi(t,0)|^4\,dt\right]^{\frac34}\|e^{it\p_x^2}\psi_0\|_{L_{[0,\infty)}^4C_x}. 
\end{align*}
In this way, we obtain the key bound forward in time:
\[
\int_0^T|\psi(t,0)|^4\,dt\leq \|e^{it\p_x^2}\psi_0\|_{L_{[0,\infty)}^4C_x}^4.
\]
Sending $T\nearrow\sup I$ and repeating this argument backward in time yields~\eqref{estimate}.
\end{proof}

Theorem \ref{t:point GWP} is a straightforward corollary of Theorem~\ref{t:LWP} and Proposition~\ref{p:identity}:

\begin{proof}[Proof of Theorem~\ref{t:point GWP}]
Combining the Strichartz estimate \eqref{linear Strichartz} and the inequality \eqref{estimate}, we see that the strong solution \(\psi\) of \eqref{point NLS} satisfies \eqref{scattering criterion} on the maximal interval of existence.   By Theorem~\ref{t:LWP} and time reversal symmetry, we deduce that the solution is global and that it scatters both backward and forward in time.
\end{proof}

\section{Global solutions and scattering for concentrated nonlinearities}\label{S:4}

In this section we prove Theorem~\ref{t:perturbations}. If it were true that \(g_\eps \to \delta\) in \(\cM\), then this result would follow from a simple refinement of Theorem~\ref{t:LWP}. Unfortunately, this convergence only holds in a weak sense.

Our starting point is to obtain strong convergence for one of the forcing terms in the difference between $\psi^\eps$ and $\psi$.  This requires the introduction of appropriate localizations in frequency or space:
\begin{lem}\label{L:4.1}
Let \(\psi\in C^{}_tL_x^2\cap L_t^4C_x^{}\) be the unique global solution to \eqref{point NLS} with initial data \(\psi(0) \in L^2(\R)\). Then, for all dyadic \(N\geq 1\), we have
\begin{align}
\lim_{\eps\to0}\left\|P_{\leq N}\int_{0}^t e^{i(t-s)\p_x^2}\Bigl[|\psi(s)|^2\psi(s)\bigl[g_\eps - \delta\bigr]\Bigr]\,ds\right\|_{C^{}_tL_x^2} &= 0,\label{limit CL2}\\
\lim_{\eps\to0}\left\| \int_{0}^t e^{i(t-s)\p_x^2}\Bigl[|\psi(s)|^2\psi(s)\bigl[g_\eps - \delta\bigr]\Bigr]\,ds\right\|_{L_t^4L^4_{g_\eps}} &= 0.\label{limit L4C}
\end{align}
Here $P_{\leq N}$ denotes the Littlewood--Paley projection to frequencies not exceeding $N$.
\end{lem}

\begin{proof}
Let \(k(t,x)\) denote the fundamental solution of the linear Schr\"odinger equation as in \eqref{fundamental solution} and let \(F(t,x) := |\psi(t,x)|^2\psi(t,x)\) as in \eqref{F(t,x)}.  As \(\int_\R g_\eps\,dx = 1\), we may write
\begin{align*}
&\int_{0}^t e^{i(t-s)\p_x^2}\Bigl[|\psi(s)|^2\psi(s)\bigl[g_\eps - \delta\bigr]\Bigr]\,ds\\
&\qquad = \underbrace{\int_{0}^t\int_\R k(t-s,x-y)\bigl[F(s,y) - F(s,0)\bigr]g_\eps(y)\,dy\,ds}_{=:I_\eps(t,x)}\\
&\qquad\quad + \underbrace{\int_{0}^t\int_\R \bigl[ k(t-s,x-y) - k(t-s,x) \bigr] F(s,0) g_\eps(y)\,dy\,ds}_{=:\II_\eps(t,x)}.
\end{align*}

We begin with the contribution of \(I_\eps(t,x)\) to both \eqref{limit CL2} and \eqref{limit L4C}.  For almost every $s$, $F(s,y)$ is a bounded and continuous function of $y$.  For such $s$,
\[
\bigl\| \bigl[F(s,y) - F(s,0)\bigr]g_\eps(y)\bigr\|_{\cM} = \int_\R \bigl|F(s,\eps y') - F(s,0)\bigr| \bigl|g(y')\bigr|\,dy' \longrightarrow 0
\]
as $\eps\to0$.  Moreover, we have the (uniform in $\eps$) domination 
\[
\bigl\| \bigl[F(s,y) - F(s,0)\bigr]g_\eps(y)\bigr\|_{\cM} \leq 2 \| g\|_{L^1} \| \psi(s)\|_{C_x}^3 \in L^{4/3}_s(\R).
\]
Thus, by the Dominated Convergence Theorem, we deduce that
\[
\bigl\| \bigl[F(s,y) - F(s,0)\bigr]g_\eps(y)\bigr\|_{L^{4/3}_s\cM}  \longrightarrow 0.
\]
The inequality \eqref{linear Strichartz} then demonstrates that \(I_\eps(t,x)\) does indeed make an acceptable contribution to both \eqref{limit CL2} and \eqref{limit L4C}.

The contributions of \(\II_\eps(t,x)\) are rather more subtle and we will treat \eqref{limit CL2} and \eqref{limit L4C} separately.  Initially, we impose the restriction that $g$ has compact support, concretely, $\supp(g)\subseteq[-R,R]$.  This restriction will be removed later. 

Let \(K_N(x) = NK(Nx)\) be the kernel of the Littlewood--Paley projection \(P_{\leq N}\).  Here \(K\in \Schwartz(\R)\) and so using the Fundamental Theorem of Calculus, we see that
\begin{align}\label{Lip LP}
 \sup_{|z|\leq \eps R} \ \int_\R \bigl|K_N(y-z) - K_N(y) \bigr| \,dy & \lesssim \eps N R.
\end{align}
 
After a change of variables, can write the low-frequency portion of $\II_\eps$ as follows:
\begin{align*}
&P_{\leq N}\II_\eps(t,x) = \int_\R\int_{0}^t\int_\R k(t-s,x-y) \bigl[K_N(y-z) - K_N(y)\bigr]F(s,0)g_\eps(z)\,dz\,ds\,dy.
\end{align*}
Recalling the Strichartz estimate \eqref{linear Strichartz}, \eqref{Lip LP}, and the support restriction on $g$, we can immediately deduce
\begin{align*}
\|P_{\leq N}\II_\eps\|_{C^{}_tL_x^2}\lesssim \eps N R \|\psi\|_{L_t^4C_x^{ }}^3\|g\|_{L^1} \quad\text{uniformly for $0<\eps\leq 1$.}
\end{align*}
As this converges to zero as $\eps\to0$, we have now completed the proof of \eqref{limit CL2} for compactly supported \(g\).

To bound the contribution of \(\II_\eps\) to \eqref{limit L4C}, we introduce
\begin{align*}
h_\eps(t,s,x) &= \bbo_{\{0<s<t\}\cup\{t<s<0\}}\int_\R\bigl[k(t-s,x-y) - k(t-s,x)\bigr]F(s,0)g_\eps(y)\,dy,\\
H(t,s) &= \bbo_{\{0<s<t\}\cup\{t<s<0\}}\tfrac 1{\sqrt{\pi|t-s|}}|F(s,0)|\|g\|_{L^1}.
\end{align*}
By \eqref{fundamental solution}, we have 
\begin{equation}\label{dominating}
\sup_{|x|\leq R}|h_\eps(t,s,x)|\leq H(t,s) \quad\text{for all \(t,s\in \R\) and all \(0<\eps\leq 1\).}
\end{equation}
Further, by the Hardy--Littlewood convolution inequality, we have
\begin{equation}\label{domination}
\left\| \II_\eps \right\|_{L_t^4 L^\infty_x} \leq \left\|\int_\R H(t,s)\,ds\right\|_{L_t^4}\lesssim \|\psi\|_{L_t^4C_x}^3\|g\|_{L^1},
\end{equation}
which shows that \(s\mapsto H(t,s)\) is integrable for a.e. \(t\in \R\).

As \(F\in L_t^{4/3}C_x^{ }\), so $|F(s,0)|<\infty$ for almost every $s\in\R$.  For any such $s\neq t$,
\[
	\sup_{|x|\leq R} \  \sup_{|y|\leq  R\eps}\ |k(t-s,x-y) - k(t-s,x)| |F(s,0)|\|g\|_{L^1} \to0 \quad \text{as}\quad  \eps\to0,
\]
due to the continuity of $z\mapsto k(\tau,z)$ for $\tau\neq 0$.  In this way, we see that
\[
	\sup_{|x|\leq R}\ |h_\eps(t,s,x)| \to 0 \qtq{as} \eps\to0 \quad\text{for a.e. $t,s\in\R$.}
\]

Applying the Dominated Convergence Theorem with the dominating function \(s\mapsto H(t,s)\) discussed in connection with \eqref{domination},  we deduce that
\begin{equation}\label{lim 1}
\lim_{\eps \to 0} \ \sup_{|x|\leq R}\  | \II_\eps(t,x)| \leq  \lim_{\eps \to 0}\int_\R \sup_{|x|\leq R}\  | h_\eps(t,s,x) | \,ds = 0 \quad\text{for a.e. \(t\in \R\)}
\end{equation}
and consequently,
\begin{equation}\label{lim 1'}
\lim_{\eps \to 0} \ \bigl\| \bbo_{\{|x|\leq R\}} \II_\eps(t,x) \bigr\|_{L^4_{g_\eps}}  = 0 \quad\text{for a.e. \(t\in \R\)}.
\end{equation}
Finally, using \eqref{domination}, the Dominated Convergence Theorem completes the proof of \eqref{limit L4C}, at least for compactly supported \(g\).

To finish, we remove the restriction that $g$ be of compact support. Given \(R>0\), we set \(A(R) := \int_{-R}^R g(x)\,dx\).  By hypothesis $\int_\R g =1$ and so \(A(R)\to 1\) as \(R\to\infty\). For $R$ sufficiently large, \(A(R)>0\) and we may decompose
\begin{align*}
&\int_0^t e^{i(t-s)\p_x^2}\Bigl[|\psi(s)|^2\psi(s)\bigl[g_\eps - \delta\bigr]\Bigr]\,ds\\
&\qquad = A(R)\int_0^t e^{i(t-s)\p_x^2}\Bigl[|\psi(s)|^2\psi(s)\bigl[\tfrac1{A(R)}g_\eps\bbo_{\{|x|\leq R\eps\}} - \delta\bigr]\Bigr]\,ds\\
&\qquad \quad + \int_0^t e^{i(t-s)\p_x^2}\Bigl[|\psi(s)|^2\psi(s)\,g_\eps\bbo_{\{|x|>R\eps\}}\Bigr]\,ds\\
&\qquad\quad + \bigl[1 - A(R)\bigr]\int_0^t e^{i(t-s)\p_x^2}\Bigl[|\psi(s)|^2\psi(s)\,\delta\Bigr]\,ds.
\end{align*}
Our earlier argument demonstrates that for any choice of $R$, the contribution of the first summand to \eqref{limit CL2} and \eqref{limit L4C} converges to zero as $\eps\to0$. For the second and third summands, we may apply \eqref{linear Strichartz} to obtain the \(\eps\)-independent estimates
\begin{align*}
\left\|\int_0^t e^{i(t-s)\p_x^2}\Bigl[|\psi(s)|^2\psi(s)\,g_\eps\bbo_{\{|x|>R\eps\}}\Bigr]\,ds\right\|_{C^{}_tL_x^2\cap L_t^4C_x}
	&\lesssim \|\psi\|_{L_t^4C_x}^3\|g\|_{L^1(|x|>R)}\\
\left\|\bigl[1 - A(R)\bigr]\int_0^t e^{i(t-s)\p_x^2}\Bigl[|\psi(s)|^2\psi(s)\,\delta\Bigr]\,ds\right\|_{C^{}_tL_x^2\cap L_t^4C_x}
	&\lesssim |1 - A(R)| \|\psi\|_{L_t^4C_x}^3.
\end{align*}
Consequently, by first sending \(\eps \to 0\) and then taking \(R\to\infty\), we obtain \eqref{limit CL2} and \eqref{limit L4C} for all \(g\in L^1\).
\end{proof}

\begin{lem}\label{L:>N}
Let $\psi(t)$ be the unique global solution to \eqref{point NLS} with initial data \(\psi(0) \in L^2(\R)\).  Then
\begin{equation}\label{global equicontinuity}
\lim_{N\to\infty}\|P_{>N}\psi\|_{C^{}_tL_x^2} = 0.
\end{equation}
\end{lem}
\begin{proof}
By Theorem~\ref{t:point GWP}, the map \([-\infty,\infty]\ni t\mapsto e^{-it\p_x^2}\psi(t)\in L^2(\R)\) is continuous, where \([-\infty,\infty]\) denotes the \(2\)-point compactification of \(\R\). It follows that the set \(\{e^{-it\p_x^2}\psi(t):t\in \R\}\) is precompact in \(L^2(\R)\) and therefore tight in frequency. As \(e^{-it\p_x^2}\) does not modify the frequency distribution, the claim \eqref{global equicontinuity} follows.
\end{proof}

We are now in a position to complete the
\begin{proof}[Proof of Theorem~\ref{t:perturbations}]
We will work forward in time; the result then carries over to all times via time-reversal symmetry.

Theorem~\ref{t:point GWP} shows that the solution $\psi(t)$ to \eqref{point NLS} is global and satisfies \eqref{STB:1}.  Thus, given $\eta>0$, which will be chosen shortly, we may partition $[0,\infty)$ into intervals
\[
J_m = [t_{m-1},t_m) \qtq{with}  0 = t_0 < t_1 < \dots < t_M = \infty,
\]
so that 
\begin{equation}\label{big ptn}
\max_{m=1,\dots, M}\|\psi\|_{L_{J_m}^4C_x}\leq \eta \qtq{and} M\leq \tfrac1{\eta^4}\|\psi\|_{L_t^4C_x}^4+1.
\end{equation}

Let $\psi^\eps(t)$ denote the maximal solution to \eqref{g epsilon NLS} with initial data $\psi^\eps(0)$; recall that $\psi^\eps(0)\to \psi(0)$ in $L^2(\R)$.
Unlike $\psi$, we do not know if $\psi^\eps$  will be global.  Nevertheless, Theorem~\ref{t:LWP}\ref{LWP:SC} shows that finite time blowup must be accompanied by infinite spacetime norm.  Consequently, for each $\eps>0$, we may define a `stopping time' $\tau^\eps\in[0,\infty]$ as the maximal time for which
\begin{equation}\label{not stopped'}
\|\psi^\eps - \psi\|_{L_{[0,\tau)}^4 L_{g_\eps}^4}\leq \| g \|_{L^1}^{\frac14} \eta.
\end{equation}

Given a dyadic $N$, $t\in[0,\infty]$, and a spacetime function $u(s,x)$, we define
\begin{equation}\label{err'}
\Err_N^{\,\eps} (u, t):= \|P_{\leq N}u \|_{C_{[0,t)}L_x^2} + \|u\|_{L_{[0,t)}^4L_{g_\eps}^4} .
\end{equation}
Our intention is to bound this with $u=\psi^\eps-\psi$.

For any $0\leq m < M$ and $0<t\leq t_{m+1}\wedge \tau^\eps$, \eqref{Duhamel} allows us to write 
\begin{align*}
\psi^\eps(t) - \psi(t) &= \underbrace{e^{i t\p_x^2}\bigl[\psi^\eps(0) - \psi(0)\bigr]}_{=:\I}\\
&\quad -\underbrace{i\int_{0}^t e^{i(t-s)\p_x^2}\Bigl[|\psi(s)|^2\psi(s)\bigl[g_\eps - \delta\bigr]\Bigr]\,ds}_{=:\II} \\
&\quad - \underbrace{i \int_{0}^t e^{i(t-s)\p_x^2}\Bigl[\bigl[|\psi^\eps(s)|^2\psi^\eps(s) - |\psi(s)|^2\psi(s)\bigr]g_\eps\Bigr] \bbo_{[0,t_m)}(s) \,ds}_{=:\III_m^{<}} \\
&\quad - \underbrace{i \int_{0}^t e^{i(t-s)\p_x^2}\Bigl[\bigl[|\psi^\eps(s)|^2\psi^\eps(s) - |\psi(s)|^2\psi(s)\bigr]g_\eps\Bigr] \bbo_{[t_m,t_{m+1})}(s) \,ds}_{=:\III_m^{>}}\,.
\end{align*}

By \eqref{linear Strichartz} and our hypothesis on $\psi^\eps(0)$,
\begin{align}\label{I bnd}
\Err_N^{\,\eps} (\I,\infty) \lesssim \|\psi^\eps(0) - \psi(0)\|_{L^2} \to 0 \qtq{as} \eps \to 0.
\end{align}
Combining \eqref{limit CL2} and \eqref{limit L4C}, we obtain
\begin{align}\label{II bnd}
\lim_{\eps \to 0} \ \Err_N^{\,\eps} \bigl( \II, \infty \bigr) =0 \qtq{for each $N\in 2^\N$.}
\end{align}

We now focus on the remaining terms.  For any time interval $J\subseteq [0,\tau^\eps)$, the Strichartz inequality and elementary manipulations yield
\begin{align*}
\biggl\| \int_0^t & e^{i(t-s)\p_x^2} \Bigl[\bigl[|\psi^\eps(s)|^2\psi^\eps(s) - |\psi(s)|^2\psi(s)\bigr]g_\eps\Bigr] \bbo_J(s) \,ds\biggr\|_{{C^{}_tL_x^2\cap L_t^4C_x}} \\
	&\lesssim \|g\|_{L^1}^{\frac14}
	\bigl[\|\psi\|_{L_{J}^4L_{g_\eps}^4}^2 + \|\psi^\eps - \psi\|_{L_{J}^4L_{g_\eps}^4}^2\bigr]\|\psi^\eps - \psi\|_{L_{J}^4L_{g_\eps}^4}.
\end{align*}

Taking $J=[0,t_m)$ and using \eqref{not stopped'}, this yields
\begin{align}\label{III bnd}
\Err_N^{\,\eps} ( \III_m^{<}, t_{m+1}\wedge \tau^\eps)
		\lesssim  \|g\|_{L^1}^{\frac34} \bigl[\|\psi\|_{L_t^4C_x^{ }}^2 + \eta^2\bigr] \Err_N^{\,\eps}( \psi^\eps - \psi, t_{m}\wedge\tau^\eps).
\end{align}
By taking $J=[t_m,t_{m+1})$ and also employing \eqref{big ptn} and \eqref{trivial/useful}, we obtain
\begin{align}\label{III' bnd}
\Err_N^{\,\eps} ( \III_m^{>}, t_{m+1}\wedge \tau^\eps)  &\lesssim  \|g\|_{L^1}^{\frac34} \eta^2 \Err_N^{\,\eps}( \psi^\eps - \psi, t_{m+1}\wedge\tau^\eps) .
\end{align}

Putting all the pieces together, we deduce that
\begin{align*}
\Err_N^{\,\eps} ( \psi^\eps - \psi, t_{m+1}\wedge\tau^\eps)  &\lesssim \Err_N^{\,\eps} (\I,\infty) + \Err_N^{\,\eps} (\II,\infty)\\
	&\quad +  \|g\|_{L^1}^{\frac34} \bigl[\|\psi\|_{L_t^4C_x^{ }}^2 + \eta^2\bigr] \Err_N^{\,\eps}( \psi^\eps - \psi, t_{m}\wedge\tau^\eps)\\
	&\quad + \|g\|_{L^1}^{\frac34} \eta^2 \Err_N^{\,\eps}( \psi^\eps - \psi, t_{m+1}\wedge\tau^\eps) .
\end{align*}
By choosing $\eta$ sufficiently small, depending only on $g$, we may absorb the last term into the left-hand side.  This sets the stage for a simple induction on $0\leq m < M$.  For the base step, $m=0$, we employ
$$
\Err_N^{\,\eps} ( \psi^\eps - \psi, 0\wedge \tau^\eps)\leq \|\psi^\eps(0) - \psi(0)\|_{L^2}.
$$
Proceeding inductively, we deduce 
\begin{align*}
\Err_N^{\,\eps} ( \psi^\eps - \psi, \tau^\eps)  &\leq C\bigl( \|\psi\|_{L_t^4C_x^{ }}, \|g\|_{L^1} \bigr)
	\bigl[ \|\psi^\eps(0) - \psi(0)\|_{L^2} + \Err_N^{\,\eps} (\I,\infty) \\
	&\qquad\qquad\qquad\qquad \qquad \qquad +\Err_N^{\,\eps} (\II,\infty)\bigr].
\end{align*}
Together with \eqref{I bnd} and \eqref{II bnd}, this yields
\begin{equation}\label{err to zero}
\lim_{\eps\to 0} \ \Err_N^{\,\eps} (\psi^\eps - \psi, \tau^\eps) = 0 \qtq{for each} N\in 2^\N .
\end{equation}
In particular,
$$
\|\psi^\eps - \psi\|_{L_{[0,\tau^\eps)}^4L_{g_\eps}^4} \to 0   \qtq{as}  \eps \to 0.
$$
Comparing this with \eqref{not stopped'}, we see that if $\eps$ is sufficiently small, a continuity argument yields $\tau^\eps=\infty$.  Therefore, the solution $\psi^\eps$ is global, satisfies \eqref{scattering criterion}, and so scatters.

It remains only to verify \eqref{g converg}.  The defect of \eqref{err to zero} in this regard is that it does not cover high frequencies.  Using conservation of mass, we find
\begin{align*}
\|\psi^\eps - \psi\|_{C^{}_tL_x^2}^2 &= \sup_{t\in \R} \Bigl[\|\psi^\eps(t)\|_{L^2}^2 - \|\psi(t)\|_{L^2}^2 - 2\Re\< \psi^\eps(t) - \psi(t),\psi(t)\>\Bigr]\\
&\leq \|\psi^\eps(0)\|_{L^2}^2 - \|\psi(0)\|_{L^2}^2 + 2\|P_{\leq N}(\psi^\eps - \psi)\|_{C^{}_tL_x^2}\|\psi(0)\|_{L^2}\\
&\quad + 2\bigl[\|\psi^\eps(0)\|_{L^2} + \|\psi(0)\|_{L^2}\bigr]\|P_{>N}\psi\|_{C^{}_tL_x^2}.
\end{align*}
Lemma~\ref{L:>N} shows that the final term can be made small by choosing $N$ large.  Using $\psi^\eps(0)\to\psi(0)$ and \eqref{err to zero}, we see that the remaining terms converge to zero as $\eps\to 0$, thereby completing the proof of \eqref{g converg}.
\end{proof}

\bibliographystyle{habbrv}
\bibliography{refs}

\end{document}